\documentclass[a4paper]{amsart}
\usepackage[latin1]{inputenc}
\usepackage{amssymb}
\usepackage{amsmath}
\usepackage{mathrsfs}
\usepackage{eufrak}
\usepackage{amsthm}
\usepackage{amsfonts}
\usepackage{textcomp}
\usepackage{graphicx}
\usepackage[pdftex]{color}
\usepackage{paralist}
\usepackage[shortlabels]{enumitem}
\usepackage{hyperref}
\usepackage{comment}
\usepackage[arrow, matrix, curve]{xy}

\newtheorem*{corollary*}{Corollary}
\newtheorem{theorem}{Theorem}[section]
\newtheorem*{theorem*}{Theorem}

\newtheorem{corollary}[theorem]{Corollary}
\newtheorem{lemma}[theorem]{Lemma}
\newtheorem{proposition}[theorem]{Proposition}

\newtheorem*{claim*}{Claim}

\theoremstyle{definition}

\newtheorem{example}[theorem]{Example}

\theoremstyle{remark}

\numberwithin{equation}{theorem}

\makeatletter
\renewcommand*\env@matrix[1][\
arraystretch]{%
  \edef\arraystretch{#1}%
  \hskip -\arraycolsep
  \let\@ifnextchar\new@ifnextchar
  \array{*\c@MaxMatrixCols c}}
\makeatother

\renewcommand{\mod}{\operatorname{\mathsf{mod\text{-}}}}

\newcommand{\Hom}{\operatorname{Hom}}

\renewcommand{\top}{\operatorname{\mathrm{top}}}

\newcommand{\soc}{\operatorname{\mathrm{soc}}}
\newcommand{\findim}{\operatorname{findim}}
\newcommand{\gldim}{\operatorname{gldim}}
\newcommand{\id}{\operatorname{id}}
\newcommand{\pd}{\operatorname{pd}}

\begin{document}

\title{On bounds of homological dimensions in Nakayama algebras}
\date{\today}



\author{Dag Oskar Madsen}
\address{Faculty of education and arts, Nord University, Post box 1490, NO-8049 Bod{\o}, Norway}
\email{dag.o.madsen@nord.no}

\author{Ren\'{e} Marczinzik}
\address{Institute of algebra and number theory, University of Stuttgart, Pfaffenwaldring 57, 70569 Stuttgart, Germany}
\email{marczire@mathematik.uni-stuttgart.de}

\begin{abstract}
Let $A$ be a Nakayama algebra with $n$ simple modules and a simple module $S$ of even projective dimension $m$. Choose $m$ minimal such that a simple $A$-module with projective dimension $2m$ exists, then we show that the global dimension of $A$ is bounded by $n+m-1$.
This gives a combined generalisation of results of Gustafson \cite{Gus} and Madsen \cite{Mad}.
In \cite{Bro}, Brown proved that the global dimension of quasi-hereditary Nakayama algebras with $n$ simple modules is bounded by $n$. Using our result on the bounds of global dimensions of Nakayama algebras, we give a short new proof of this result and generalise Brown's result from quasi-hereditary to standardly stratified Nakayama algebras, where the global dimension is replaced with the finitistic dimension.
\end{abstract}

\maketitle
\section*{Introduction}
We always assume that our algebras are finite dimensional over a field $K$ and furthermore they are connected and non-semisimple if nothing is stated otherwise. We assume that all modules are finite dimensional right modules if nothing is stated otherwise.
Nakayama algebras are defined as algebras such that every indecomposable projective left or right module is uniserial.
See for example the books \cite{ARS}, \cite{SkoYam} and \cite{Zi} for sections on the basics and importance of Nakayama algebras.

In \cite{Gus}, Gustafson showed that a Nakayama algebra with $n$ simple modules and finite global dimension has global dimension at most $2n-2$ and in \cite{Mad}, Madsen showed that a Nakayama algebra has finite global dimension if and only if it has a simple module of even projective dimension. We combine and generalise those two results to the following, which is our first main result:
\begin{theorem*}
Let $A$ be a Nakayama algebra with a simple module $S$ of even projective dimension. Choose $m$ minimal such that a simple $A$-module has projective dimension equal to $2m$. Then the global dimension of $A$ is bounded by $n+m-1$.
\end{theorem*}

In forthcoming work we discuss whether the bounds $n+m-1$ are optimal and attained, where we discover a connection with the classification of Nakayama algebras that are higher Auslander algebras.
Recall that the finitistic dimension $\findim(A)$ of an algebra $A$ is defined as the supremum of all projective dimensions of modules having finite projective dimension. It is one of the most famous conjectures in the representation theory of finite dimensional algebras, whether the finitistic dimension is always finite.

In \cite{Bro}, Brown showed that the global dimension of quasi-hereditary Nakayama algebras with $n$ simple modules is bounded by $n$.
Standardly stratified algebras were introduced as a generalisation of quasi-hereditary algebras. Indeed, it was proven in \cite{AHLU} that a standardly stratified algebra is quasi-hereditary iff it has finite global dimension.
In recent years there was much interest to give equalities and inequalities for the finitistic dimension of standardly startified algebras, see for example \cite{MazOv}, \cite{AHLU2}, \cite{Mar2} and \cite{Maz}.

We prove our second main result for standardly stratified Nakayama algebras:
\begin{theorem*}
Let $A$ be a standardly stratified Nakayama algebra with $n$ simple modules. Then $\findim(A) \leq n$.
\end{theorem*}
The proof uses our first main result on the bounds of the global dimension to give a short proof for quasi-hereditary Nakayama algebras.
We then classifiy the standardly stratified Nakayama algebras of infinite global dimension and look at their finitistic dimension to obtain our second main result.
We remark that a proof of Proposition \ref{propprojdim} was also obtained by Aaron Chan with more elementary methods. The second named author thanks Aaron Chan for useful discussions. We profited from the GAP package QPA, see \cite{QPA}, to calculate numerous examples.

\section{Preliminaries}
\subsection{General preliminaries}
Throughout $A$ is a finite dimensional, non-semisimple and connected algebra over a field $K$. We always work with finite dimensional right modules, if not stated otherwise. By $\mod A$, we denote the category of finite dimensional right $A$-modules and $J$ denotes the Jacobson radical of an algebra $A$.
As usual, $D:=\Hom_K(-,K)$ denotes the $K$-duality of an algebra $A$ over the field $K$.
For background on representation theory of finite dimensional algebras and their homological algebra, we refer to \cite{ARS} and \cite{SkoYam}.
An algebra is called \emph{basic} in case the regular module does not contain a projective module of the form $P^2$ as a direct summand for an indecomposable projective module $P$. Every algebra is Morita equivalent to a basic algebra and we thus assume that all our algebras are basic if nothing is stated otherwise. Note that all the homological notions in this text are invariant under Morita equivalence and thus it is no restriction on the generality of our results to assume that our algebras are basic.
For a fixed set of primitive orthogonal idempotents $e_1,e_2 ,\dots, e_n$ with $1= e_1 + e_2 + \dots + e_n$, we denote by $S_i=e_iA/e_iJ$, $P_i=e_i A$ and $I_i=D(Ae_i)$  the simple, indecomposable projective and indecomposable injective module, respectively, corresponding to the primitive idempotent $e_i$, for $1 \leq i \leq n$. The \emph{finitistic dimension} $\findim(A)$ of an algebra $A$ is defined as the supremum of all projective dimensions of modules having finite projective dimension.
The \emph{global dimension} is defined as the supremum of all projective dimensions of modules. Thus the global dimension coincides with the finitistic dimension in case the global dimension is finite.

An algebra $A$ is called \emph{Nakayama algebra} in case every indecomposable left or right module is uniserial. We refer to \cite{AnFul} and \cite{SkoYam} for results on Nakayama algebras which we collect in the following without proof.
A Nakayama algebra either has no simple projective module or it has a unique simple projective module. In the last case the algebra is triangular and hence the global dimension is bounded by $n-1$. If the Nakayama algebra is not triangular and has $n$ simple modules, it is possible to order the primitive idempotents such that there are projective covers $$e_iA \to e_{i-1}J,$$ for $2 \leq i \leq n$, and a projective cover $$e_1A \to e_nJ.$$ Fix such an order $e_1,e_2 ,\dots, e_n$ of a primitive orthogonal idempotents. This order is uniquely defined up to a cyclic permutation. We also have $\tau(S_i) \cong S_{i+1}$, for $1 \leq i \leq n-1$, and $\tau(S_n) \cong S_1$, where $\tau$ denotes the Aulander-Reiten translate. The Nakayama algebra $A$ is uniquely determined by the length $c_i$ of the indecomposable projective modules $e_i A$. The sequence $[c_1, c_2, \dots, c_n]$ is called the \emph{Kupisch series} for $A$. One can show that $c_{i+1} \geq c_i -1$, for all $2 \leq i \leq n$, and $c_1 \geq c_n-1$. Conversely, any sequence of integers greater or equal than $2$ satisfying those requirements is the Kupisch series for some Nakayama algebra.  We look at the indices $i$ of the $c_i$ modulo $n$ so that $c_i$ is defined for all $i \in \mathbb{Z}$. A Nakayama algebra is selfinjective if and only if its Kupisch series is constant. In case the algebra is not selfinjective, then after a cyclic reordering of the indices one can always get $c_1 = c_n-1$ with $c_1$ minimal among the $c_i$.
Every indecomposable module of a Nakayama algebra is isomorphic to a module of the form $e_i A/e_i J^k$.
For explicit calcutions of minimal projective resolutions or injective coresolutions in Nakayama algebras, see for example \cite{Mar}.

The following lemma is a direct consequence of the classification of Nakayama algebras by their diagrams, see chapter 10.3. of \cite{DK} (note that in this textbook Nakayama algebras are called serial algebras). In Theorem 10.3.1. of \cite{DK} it is proven that the diagram of a Nakayama algebra is either a directed line or a directed cycle.
\begin{lemma}\label{naklemma}
Let $A$ be a connected Nakayama algebra having no simple projective module, and let $e \in A$ be a primitive idempotent. Then $A/AeA$ is a connected Nakayama algebra of finite global dimension. The algebra $eAe$ is also a connected Nakayama algebra which is semisimple if and only if the length of $eA$ as an $A$-module is less or equal than $n$. If the algebra $eAe$ is not semisimple, then $A/AeA$ is hereditary.
\end{lemma}

\subsection{Preliminaries on standardly stratified algebras}
See \cite{DR} for an introduction to quasi-hereditary algebras and \cite{Rei}, \cite{ADL}, \cite{FrMa}, for the basics of standardly stratified algebras. We just briefly recall the most important definitions.
Let $(A,E)$ be an algebra together with an ordered complete sequence of primitive orthogonal idempotents $E=(e_1,e_2,\dots,e_n)$. Then the sequence of standard right $A$-modules is defined by $\Delta=(\Delta(1),\dots,\Delta(n))$, where $\Delta(i)=e_iA/e_iJ(e_{i+1}+e_{i+1}+\dots+e_n)A$ for $1 \leq i \leq n$ with $\Delta(n)=e_n A$.
The sequence of proper standard right $A$-modules $\bar{\Delta}$ is defined by $\bar{\Delta}=(\bar{\Delta}(1),\dots,\bar{\Delta}(n))$, where $\bar{\Delta}(i)=e_iA/e_iJ(e_i+e_{i+1}+\dots+e_n)A$ for $1 \leq i \leq n$. Dually, one can define left standard modules $\Delta^{o}(i)$ and left proper standard modules $\bar \Delta^{o}(i)$. The costandard modules $\nabla(i)$ and proper costandard modules $\bar{\nabla}(i)$ are then defined as the modules $D(\Delta^{o}(i))$ and $D(\bar \Delta^{o}(i))$. For a set of modules $C$, let $\mathcal F(C)$ be the full subcategory of $\mod A$ of all modules $M$ with a filtration $0 \subseteq M_s \subseteq \dots \subseteq M_1=M$, such that every subquotient is isomorphic to an object in $C$. A module is called \emph{proper standardly filtered} in case $M \in \mathcal F(\bar{\Delta})$ with $\bar{\Delta}:= \{ \bar{\Delta}(1), \dots, \bar{\Delta}(n) \}$. The algebra $A$ is called \emph{standardly stratified} in case $A \in \mathcal F(\bar{\Delta})$. In view of \cite{AHLU}, Theorem 2.4, we define $A$ to be \emph{quasi-hereditary}, in case it is standardly stratified and has finite global dimension. In case $A$ and $A^{\mathrm{op}}$ are both standardly stratified, then $A$ is called \emph{properly stratified} (here we use a characterisation of properly stratified algebras found in \cite{Rei} after theorem 3.6.).

We remark that if a Nakayama algebra is standardly stratified, the order of primitive idempotents defining the Kupisch series does not have to coincide with the sequence $E$ defining the standardly stratified structure.

The next proposition collects several results from the literature that we will need in this article.

\begin{proposition} \label{proplemma}
Let $A$ be a finite dimensional algebra.
\begin{enumerate}
\item Let $e$ be a primitive idempotent of an algebra $A$ such that $AeA$ is projective as a right $A$-module. In case an $A$-module $M$ has finite projective dimension, then the $eAe$-module $Me$ is projective.
\item Let $A$ be a standardly stratified algebra. Then $Ae_n A$ is projective as a right $A$-module, and $A/Ae_nA$ is again standardly stratified.
\item A Nakayama algebra with no simple projective module is quasi-hereditary if and only if there is a simple module of projective dimension equal to two.
\item Let $e$ be an idempotent such that $AeA$ is projective as a right $A$-module and let $X$ be an $A/AeA$-module.
Then $\pd_A(X) \leq \pd_{A/AeA}(X)+1$.
\item  Let $e$ be an idempotent such that $AeA$ is projective as a right $A$-module and the algebra $eAe$ is semisimple.
Then $\gldim(A) \leq \gldim(A/AeA)+2$.
\item Let $e$ be a primitive idempotent such that $AeA$ is projective as a right $A$-module. Suppose $A/AeA$ has finitistic dimension equal to $k$. Then $A$ has finitistic dimension at most $k+2$.
\end{enumerate}
\end{proposition}

\begin{proof}
\begin{enumerate}
\item See \cite{AHLU2}, Lemma 2.4.
\item See \cite{AHLU2} in the part above Proposition 1.1.
\item This is part of Proposition 3.1 of \cite{UY}.
\item This is a special case of Lemma 5.8 of \cite{APT}.
\item This is a special case of Theorem 5.4 of \cite{APT}.
\item See \cite{AHLU2}, Theorem 2.2.
\end{enumerate}
\end{proof}

\section{Upper bounds of the global dimension for Nakayama algebras}

In this section $A$ will always denote a Nakayama algebra.

Suppose $A$ has no simple projective module.
Let $\mathcal S$ denote a complete set of representatives of the isomorphism classes of simple $A$-modules. Following \cite{Mad}, we define a function $\Psi \colon \mathcal S \to \mathcal S$ by
$\psi (S) \cong (\tau^{-1})^{w(S)} S$, where $w(S)$ is the length of the injective envelope of $S$, for each $S \in \mathcal S$. We say that $S \in \mathcal S$ is \emph{$\psi$-regular} if $\psi^r(S)=S$ for some $r \geq 1$. Dual definitions were earlier considered in \cite{Gus}.

In \cite{Mad} we find the following criteria for finite global dimension.

\begin{theorem}[{\cite[Theorem 3.3]{Mad}}] \label{old}
Let $A$ be a Nakayama algebra having no simple projective module. The following are equivalent.
\begin{itemize}
    \item[(a)] $A$ has finite global dimension.
    \item[(b)] The set of $\psi$-regular simple $A$-modules is exactly the set of simple $A$-modules with even projective dimension, and $\psi$ is a cyclic permutation on this set.
    \item[(c)] There is a simple $A$-module with even projective dimension.
\end{itemize}
\end{theorem}

Denote by $\mathcal S^\psi \subseteq \mathcal S$ the set of $\psi$-regular simple $A$-modules.

\begin{theorem} \label{mainresult}
Let $A$ be a Nakayama algebra with a simple module $S$ of even projective dimension. Choose $m$ minimal such that a simple $A$-module has projective dimension equal to $2m$. Then the global dimension of $A$ is bounded by $n+m-1$.
\end{theorem}

\begin{proof}
If $m=0$, then $A$ is triangular and the global dimension is bounded by $n-1$.

Suppose $m >0$. By Theorem \ref{old}, the global dimension of $A$ is finite. We have $$\gldim A=\max_{S \in \mathcal S} \id(S) = \max_{S \in \mathcal S} \pd(S)$$ and also $$\gldim A \leq \max_{S \in \mathcal S^\psi} \pd(S)+1.$$ This inequality follows from Theorem \ref{old} and the fact that if the global dimension of $A$ is $g$, then there exist a simple module of projective dimension $g$ and a simple module of projective dimension $g-1$.

Let $d$ be the number of simple $A$-modules that are not $\psi$-regular. Then for any $S \in \mathcal S$ we have that $\psi^d(S)$ is $\psi$-regular, and hence by dualising the main argument from \cite{Gus} we get $\id (S) \leq 2d$. So $$\gldim A \leq 2d.$$

Let $S'$ be a simple module with $\pd(S')=2m$. Then $\mathcal S^\psi=\{S', \psi(S'), \dots, \psi^{n-d-1}(S') \}$. It follows from repeated use of \cite[Proposition 3.2(b)]{Mad} that $\max_{S \in \mathcal S^\psi} \pd(S) \leq 2m+2(n-d-1)=2m+2n-2d-2$. Hence $$\gldim A \leq 2m+2n-2d-1.$$

Adding the two inequalities together, we get $$2 \cdot \gldim A \leq 2m+2n-1.$$ Since $\gldim A$ must be an integer, we conclude that $$\gldim A \leq m+n-1.$$
\end{proof}

We will discuss whether the bounds $n+m-1$ for given $m$ are attained in forthcoming work, where this is related to the classification of higher Auslander algebras with high global dimension inside the class of Nakayama algebras.
We just give one example for  $m=1$.

\begin{example} \label{gldimnexample}
Let $A$ be the Nakayama algebra with Kupisch series $[2,2,2,\dots,2,3]$, which has $n$ simple modules and all but one indecomposable projective module have length two.
Then the simple module $S_i$ has projective dimension $n-i+1$, for $1 \leq i \leq n$, and hence the global dimension is $n$. The simple module $S_{n-1}$ has projective dimension $2$, and thus $m=1$. As a consequence, the bound $n+m-1$ is attained in case $m=1$.
\end{example}

\section{Standardly stratified Nakayama algebras and their finitistic dimension}
This section gives bounds on the finitistic dimension of standardly stratified Nakayama algebras. We can assume that all algebras involved are not selfinjective as the next proposition shows. We note that the next proposition is a generalisation of the main result in \cite{AC}, where the authors proved the same result with the additional assumption that $A$ is a Nakayama algebra. We note that it seems that the authors in \cite{AC} forgot to look at the local case, but their argument works with nearly the same proof in the general case.

\begin{lemma}
Let $A$ be a selfinjective algebra. Then $A$ is standardly stratified if and only if $A$ is local. $A$ is never quasi-hereditary.
\end{lemma}

\begin{proof}
Since the first syzygy $\Omega^{1}$ is an equivalence on the stable module category $\underline{\mod A}$ (see for example section IV.8 in \cite{SkoYam}), every indecomposable non-projective module $M$ has infinite projective dimension. But every standard module $\Delta(i)$, $1 \leq i \leq n$, has finite projective dimension by Proposition 1.3 of \cite{PR}, and thus every $\Delta(i)$ is projective and hence also injective because $A$ is selfinjective. By the definition of standardly stratified algebras, $\Delta(i)=e_iA/(e_iJ \epsilon_{i+1}A)$ with $\epsilon_{i+1}:=e_{i+1}+e_{i+2}+ \dots +e_n$. Thus $\Delta(i)=e_iA/(e_iJ \epsilon_{i+1}A)$ is projective for every $1 \leq i \leq n$ if and only if $(e_iJ \epsilon_{i+1}A)=0$ for every $1 \leq i \leq n$.

Assume the algebra has at least two simple modules and take $i=1$. Then the condition $(e_1J \epsilon_{2}A)=0$ together with our assumption that $A$ is connected implies that $e_1 Je_i=0$ for all $i >1$ and that there is a $j >1$ with $e_j J e_1 \neq 0$. Now this implies that $e_1A$ has socle isomorphic to the socle of $D(A e_1)$ and also $e_j A$ has socle isomorphic to the socle of $D(A e_1)$. (Here we use that the socle of indecomposable projective modules in a selfinjective algebra are simple.)
This is a contradiction, since a basic algebra is selfinjective if and only if there is a permutation $\pi \colon \{1, \dots, n \} \to \{1, \dots, n \}$ such that $\soc(e_i A) \cong \top(e_{\pi(i)}A)$ for all $1 \leq i \leq n$. Thus $A$ has to be local.

On the other hand, assume now that $A$ is local and selfinjective with simple module $S$. Then $\bar{\Delta}(1)=S$ and thus it is trivial that $A$ is $\bar{\Delta}$-filtered, since there is a unique simple module $S$. This implies that $A$ is standardly stratified.
As a selfinjective algebra, $A$ has always infinite global dimension and can thus never be quasi-hereditary. (Recall that we do assume that $A$ is not semisimple in our article.)
\end{proof}

Since local Nakayama algebras are selfinjective, and local algebras in any case are of finitistic dimension zero, we assume from now on that our algebras have at least two simple modules.

\begin{proposition} \label{propprojdim}
Let $A$ be a connected Nakayama algebra having no simple projective module and having at least two simple modules.
Let $A$ be standardly stratified but not quasi-hereditary, and assume that $A$ is not selfinjective. Then there is a simple module of infinite projective dimension, and all other simple modules have projective dimension equal to one.
\end{proposition}

\begin{proof}
Let $n$ denote the number of simple $A$-modules and assume that $A$ is standardly stratified.
Let $e_1,e_2, \dots e_n$ be an ordering of the primitive idempotents for $A$ to be standardly stratified. By definition $\bar{\Delta}(n)=e_nA/e_nJe_nA$ and $\Delta(n)=e_nA$. By Proposition \ref{proplemma} (2), we have that $Ae_n A$ is projective. Since $A$ is a Nakayama algebra, we have $e_nJe_nA = yA$ for an element $y \in e_nJe_n \setminus e_n J^{n+1} e_n$ (which is unique up to multiplication by a field element), and thus $\bar{\Delta}(n) = e_n A/yA$.
We now look at two cases.

\underline{Case 1:} Assume $\bar{\Delta}(n)=\Delta(n)$, which is equivalent to $e_nA/yA=e_nA$ or $yA=0$.
This shows that $e_nJe_n=0$, and hence the algebra $e_n A e_n$ is semisimple. We can apply Proposition \ref{proplemma} (5) and Lemma \ref{naklemma} to see that $A$ has finite global dimension and thus is quasi-hereditary.

\underline{Case 2:} Now assume that $\bar{\Delta}(n) \neq \Delta(n)$, which is equivalent to $yA \neq 0$. We show that in this case the simple module $S_n$ has infinite projective dimension. Since $e_n y e_n \neq 0$ we have that the local algebra $e_n A e_n$ is not semisimple and thus a selfinjective Nakayama algebra, since for any Nakayama algebra $A$ $eAe$ is a Nakayama algebra again for any idempotent $e$.
Let $M:=S_n$, and use Proposition \ref{proplemma} (1) to see that $M$ has infinite projective dimension, or else the simple module $Me_n$ would be projective over the local selfinjective connected algebra $e_n Ae_n$, which is impossible.
By Lemma \ref{naklemma}, the algebra $A/Ae_n A$ is hereditary.
Now Proposition \ref{proplemma} (4) gives that every simple $A$-module $S_i$ not isomorphic to $S_n$ has projective dimension at most two. But such an $S_i$ can not have projective dimension zero by assumption on $A$ and not projective dimension two since then $A$ would be quasi-hereditary and thus have finite global dimension.
So all simple $A$-modules except $S_n$ has projective dimension one, and this proves the proposition.
\end{proof}

The following proposition gives the possible Kupisch series for Nakayama algebras that are standardly stratified but not quasi-hereditary.

\begin{proposition} \label{Kupischseries}
Let $A$ be a (non-selfinjective) Nakayama algebra with $n$ simple modules and a simple module of projective dimension infinite and all other simple modules of projective dimension one. Then the Kupish series of $A$ is of the following form.
$$[k+qn,k+n-1+qn,k+n-2+qn,k+n-3+qn,\dots,k+2+qn,k+1+qn],$$
where $2 \leq k \leq n$ and $q \geq 1$, or $k=n+1$ and $q \geq 0$.
\end{proposition}

\begin{proof}
Note that in a Nakayama algebra, the simple module $S_i=e_iA/e_iJ$ has projective dimension one if and only if $e_iJ$ is projective if and only if $e_iJ \cong e_{i+1}A$ if and only if $c_i-1=c_{i+1}$ in the Kupisch series $[c_1,c_2,\dots,c_n]$ of $A$.
Since $A$ has all but one simple module of projective dimension one, the Kupisch series of $A$ has the form $[x,x+n-1,x+n-2,x+n-3,\dots,x+1]$.
Thus we can write the Kupisch series as stated. In case $q \geq 1$ or $k \geq n+1$, the Loewy length of the algebra is at least $2n$ and thus the algebra has infinite global dimension by \cite{Gus} or alternatively by showing that in this case the first simple module has infinite projective dimension. In case $q=0$ and $k \leq n$, the first simple module has projective dimension two, and thus the algebra is quasi-hereditary by Proposition \ref{proplemma} (3).
\end{proof}

We give a corollary of the previous proposition.
\begin{corollary}
Let $A$ be a standardly stratified Nakayama algebra. Then $A$ is even properly stratified.
\end{corollary}

\begin{proof}
The result is clear in case $A$ is quasi-hereditary or selfinjective, since the opposite algebra of a quasi-hereditary algebra is again quasi-hereditary in general and the opposite algebra of a local selfinjective Nakayama algebra is again a local selfinjective Nakayama algebra. Now assume that $A$ is standardly stratified with infinite global dimension and not being selfinjective. Assume $A$ has $n$ simple modules. By \ref{Kupischseries}, the Kupisch series of $A$ is of the form $[c,c+n-1,c+n-2,\dots,c+1]$ for some natural number $c \geq 2$.
By exercise 1 of chapter 32 of \cite{AnFul} the Kupisch series $[d_1,d_2,\dots,d_n]$ of the opposite algebra of a Nakayama algebra $[c_1,c_2,\dots,c_n]$ has the property that the $d_i$ are a permutation of the $c_i$.
As explained in the preliminaries we can assume that in a general Kupisch series we have $c_n=c_1+1$ with $c_1$ minimal among the $c_i$ and $c_{i+1} \geq c_i-1$.
Those conditions together with the fact that the Kupisch series of the opposite algebra is a permutation of the Kupisch series of $A$, force that the Kupisch series of the opposite algebra of $A$ coincides with the Kupisch series of $A$.
Thus $A$ is isomorphic to its opposite algebra and the result is clear.
\end{proof}

The next proposition shows that any Nakayama algebra with Kupisch series as in Proposition \ref{Kupischseries} is indeed standardly stratified. Thus this gives a classification of the standardly stratified Nakayama algebras that are not quasi-hereditary. Since it is elementary to decide whether a Nakayama algebra with a given Kupisch series is standardly stratified or not, and since we do not need the result in the following, we leave the proof of the next proposition to the interested reader.

\begin{proposition}
Let $A$ be a Nakayama algebra with Kupisch series as in Proposition \ref{Kupischseries}. Then $A$ is standardly stratified.
\end{proposition}

We can now give a proof of our second main result.
\begin{theorem}
Let $A$ be a Nakayama algebra with $n$ simple modules that is standardly stratified. Then $\findim A \leq n$. If $A$ is not quasi-hereditary, then $\findim A \leq 2$.
\end{theorem}

\begin{proof}
\underline{Case 1: Quasi-hereditary case} \newline
If a Nakayama algebra has a simple projective module, then it is triangular and hence its global dimension is bounded by $n-1$.

By Proposition \ref{proplemma} (3), we know that a Nakayama algebra having no simple projective module is quasi-hereditary if and only if there exists a simple module of projective dimension equal to two.
By Theorem \ref{mainresult}, this implies that the global dimension, which is equal to the finitistic dimension, is bounded by $n+1-1=n$.

\underline{Case 2: Infinite global dimension case} \newline
Now assume that our algebras are standardly stratified with infinite global dimension. Clearly, the result is trivial in the selfinjective case, since selfinjective algebras have finitistic dimension equal to zero.

Now assume that the algebras are additionally non-selfinjective. Then we use Proposition \ref{Kupischseries} to calculate the finitistic dimension of those algebras.
Assume that the algebra has at least two simple modules.
Let $e$ be a primitive idempotent such that $AeA$ is projective. In Proposition \ref{Kupischseries} we saw that each indecomposable projective module has length at least $n+1$, so $eAe$ is not semi-simple by Lemma \ref{naklemma}.
Also by Lemma \ref{naklemma}, the algebra $A/AeA$ is hereditary.
Thus by Proposition \ref{proplemma} (6), the finitistic dimension of $A$ is at most three.

Let $M$ be an $A$-module with $\pd(M)=3$.
From Proposition 2.2 (a) in \cite{Mad} it follows that all simple composition factors of $M$ have odd projective dimension, so they all have projective dimension one by Proposition \ref{propprojdim}. But a module cannot have larger projective dimension than the maximum pd of its composition factors, and we reach a contradiction. So in the infinite global dimension case $\findim A \leq 2$.
\end{proof}

\begin{example}
Let $n \geq 2$.
This example shows that the bounds in the previous theorem are optimal. First note that the algebra in \ref{gldimnexample} is a quasi-hereditary algebra with $n$ simple modules and global dimension $n$. Thus the finitistic dimension is equal to the global dimension since the global dimension is finite. This shows that the bound $n$ for the finitistic dimension is optimal.
The Nakayama algebra with Kupisch series $[4,5]$ is standardly stratified with infinite global dimension. The unique indecomposable injective non-projective module has projective dimension two. Thus the bound 2 for finitistic dimensions of standardly stratified Nakayama algebras with infinite global dimension is also optimal.
\end{example}

\end{document}